\documentclass[preprint,11pt]{elsarticle}
\biboptions{numbers,sort&compress}
\usepackage{amsfonts,psfrag,amsmath,mathrsfs,amsthm,color,empheq}
\usepackage{color,amssymb,bm} 
\usepackage{epsfig,soul}
\usepackage{hyperref}
\hypersetup{
	colorlinks   = true, 
	urlcolor     = blue, 
	linkcolor    = blue, 
	citecolor   = red 
}
\usepackage{multirow}

 \renewenvironment{thebibliography}[1]{%
 	\begin{oldthebibliography}{#1}%
 		\setlength{\parskip}{0.1ex}%
 		\setlength{\itemsep}{0.1ex}%
 	}%
 	{%
 	\end{oldthebibliography}%
 }

\usepackage{amsmath,mathtools} 
\usepackage{amsthm,amssymb,bbm}

\DeclareMathOperator*{\argmin}{arg\,min}
 \usepackage{graphicx,float}
 \usepackage[export]{adjustbox}
 \usepackage{newtxmath}

\usepackage{mathtools} 
 
\usepackage{setspace}
\AtBeginDocument{%
	\addtolength\abovedisplayskip{-0.2\baselineskip}%
	\addtolength\belowdisplayskip{-0.2\baselineskip}%
}
 \usepackage[margin=1in]{geometry}

\newcommand{\IR}{{\mathbb{R}}}

\newcommand{\Ii}{{\mathfrak{i}}}

\newcommand{\diag}{{\rm{diag}}}

\newcommand{\bmt}{\left[ \begin{array}{ccccccccc}}
\newcommand{\emt}{\end{array}\right]}
\newcommand{\bean}{\begin{eqnarray*}}
\newcommand{\eean}{\end{eqnarray*}}
\newcommand{\bea}{\begin{eqnarray}}
\newcommand{\eea}{\end{eqnarray}}
\newcommand{\eq}{\begin{equation}\begin{array}{lllllllll}}
\newcommand{\ee}{\end{array}\end{equation}}
\newcommand{\eqn}{\begin{equation*}\begin{array}{lllllllll}}
\newcommand{\een}{\end{array}\end{equation*}}
\newtheorem{theorem}{Theorem}[section]

\newtheorem{remark}{Remark}[section]

\allowdisplaybreaks

\graphicspath{{./Figures/}}

\journal{Elsevier}
\begin{document}

	\begin{frontmatter}
	
	\title{A Vanka-type multigrid solver for complex-shifted Laplacian systems \\from diagonalization-based parallel-in-time algorithms
	}
	
	 	\author{Yunhui He}
	 \ead{yunhui.he@ubc.ca}  	
	 \address{Department of Computer Science, The University of British Columbia, Vancouver, BC, V6T 1Z4, Canada.}

	\author{Jun Liu\corref{mycorrespondingauthor}}
	\cortext[mycorrespondingauthor]{Corresponding author}
	\ead{juliu@siue.edu}  	
	\address{Department of Mathematics and Statistics, Southern Illinois University Edwardsville, Edwardsville, IL 62026, USA.\vspace{-1cm}}

	\begin{abstract}
		
 We propose and analyze a Vanka-type multigrid solver for solving a sequence of complex-shifted Laplacian systems arising in diagonalization-based parallel-in-time algorithms for evolutionary equations. Under suitable assumption, local Fourier analysis shows the proposed Vanka-type smoother achieves a uniform smoothing factor, which is verified by several numerical examples.
	\end{abstract}
	
	
\end{frontmatter}

\section{Introduction}
Time-dependent partial differential equations (PDEs) appear ubiquitously in science and engineering,
whose numerical simulations based on the sequential time-stepping schemes are very time-consuming.
With the popularity of massively parallel processors, in addition to spatial parallelism, many parallel-in-time (PinT) algorithms \cite{gander201550} have been developed for simulating time-dependent PDEs, which can provide significant speed up over the sequential time-stepping schemes. 
We will focus on a class of ParaDIAG algorithms \cite{gander2020paradiag}, which are built upon the  diagonalization of the time discretization matrix $B=VD V^{-1}$ with $D=\diag\{\lambda_1,\lambda_2,\cdots,\lambda_n\}$ being its (complex) eigenvalues. More specifically, 
by coupling all the time steps simultaneously upon appropriate discretization in space and time, we arrive at the following all-at-once  sparse linear system (from separable PDEs)
\bea  \label{BE-A}
 	\left(B\otimes I_h+ {I_t\otimes A}\right) \bm u=
(V\otimes I_h)\left(D\otimes I_h+ {I_t\otimes A}\right)(V^{-1}\otimes I_h) \bm u= \bm f,  
\eea
where  $I_h\in\IR^{m\times m},I_t\in\IR^{n\times n}$ are identity matrices and $A\in\IR^{m\times m}$ is the spatial discretization matrix.
The diagonalization of $B$ can be guaranteed by using a boundary-value method time scheme \cite{liu2021well} or its circulant-type approximation as preconditioners \cite{McDonald2018,lin2020all,WL2020SIMAX}.
Extensive numerical results \cite{GW19,gander2020paradiag} reveal that the ParaDIAG algorithms have a very promising parallel efficiency for both parabolic \cite{lin2020all} and hyperbolic PDEs \cite{gander2019direct,WL2020SIMAX}. 
Since $\left(D\otimes I_h+ {I_t\otimes A}\right)$ is block diagonal,
the major step of ParaDIAG algorithms is to solve in parallel $n$ independent complex-shifted linear systems  ($j=1,2,\cdots,n$)
\bea \label{shifted-Laplace-system}
L_j \bm z:=(A+\lambda_j I) \bm z=\bm b_j.
\eea
Such complex-shifted linear systems are still expensive to solve by direct methods, which motivates us to develop efficient  multigrid solver.
For simplicity, we will concentrate on complex-shifted Laplacian systems with $A$ derived from the five-point stencil $\frac{1}{h^2} \begin{bsmallmatrix}&-1&\\-1&4&-1\\&-1& \end{bsmallmatrix}$ based on finite difference scheme with a mesh step size $h$ on rectangular domains, but
our method can also be extended to a general elliptic spatial differential operator with finite element scheme on irregular domain.
 
Complex-shifted Laplacian multigrid preconditioner has been extensively studied \cite{Erlangga2006,Cools2013} for preconditioning Helmholtz equations, where the main efforts are devoted to determining the optimal shift \cite{Cocquet2017} for achieving better GMRES convergence rates. 
In \cite{Hocking2021}, the authors studied optimal complex relaxation parameters minimizing smoothing factors of multigrid with damped Jacobi smoother and
red-black successive over-relaxation (SOR) smoother for solving complex-shifted linear systems, which also inspires our current work.
One major difference of our considered problem from such Helmholtz equations is that our complex-shifts are given by the eigenvalues of the time discretization matrix $B$ that depends on the time step size (rather than the wavenumbers).

In this work, we generalized the additive element-wise  Vanka smoother  \cite{CH2021addVanka}
for solving the complex-shifted Laplacian systems from ParaDIAG algorithms. 
Under suitable conditions on the spatial and time step sizes, through  local Fourier analysis (LFA) techniques we proved that the optimal smoothing factor can be (approximately) achieved with the known optimal relaxation parameter. Rather than solving restricted subproblems in classical Vanka setting, we derived the explicit stencil of the Vanka smoother which can facilitate more efficient implementation.

The paper is organized as follows. In the next section we propose and analyze an element-wise additive Vanka smoother. In Section 3, we present some numerical examples with both forward problem (heat PDE) and inverse problem (backward heat conduction problem) to verify our theoretical results. Finally, some conclusions are made in Section 4.
\section{An additive Vanka-type multigrid solver}

We adapt the additive element-wise Vanka smoother  proposed in  \cite{CH2021addVanka}  originally for the Laplacian to solve \eqref{shifted-Laplace-system} within multigrid methods. Although Vanka-type smoothers are well-studied \cite{Vanka1986,delaRiva2019,Farrell2020,Claus2021}, there seems no study of additive Vanka smoother applied to complex-shifted Laplacian yet.   The additive element-wise Vanka smoother for  \eqref{shifted-Laplace-system} has  the following parallelizable form 
\begin{equation}
	M_{e,j} := \sum_{i=1}^{N} R_i^T W_i L_{i,j}^{-1} R_i,
\end{equation} 
where $W_i$ is a weighting matrix, and $L_{i,j}$ is the coefficient matrix of $i$-th subproblem, and $R_i$ is a restriction operator mapping the global vector to the $i$-th subproblem.
The error propagation operator of relaxation scheme is  $S_j = I-\omega M_{e,j} L_j$,  where  $\omega$ is a relaxation parameter to be determined.   Here, we apply LFA to help us identify or approximate the optimal relaxation parameter $\omega$. LFA is a  powerful tool to predict and analyze multigrid convergence performance \cite{wienands2004practical,trottenberg2000multigrid}.

\subsection{Local Fourier analysis}

We consider standard coarsening,  and the low and high frequencies are given by  ${\boldsymbol{\theta}=(\theta_1,\theta_2) }\in  T^{\rm{L}} =\left[-\frac{\pi}{2}, \frac{\pi}{2}\right)^2$, $\boldsymbol{\theta} \in  T^{\rm{H}} =\left[-\frac{\pi}{2}, \frac{3\pi}{2}\right)^2 \setminus T^{\rm{L}}$, respectively.  To analyze multigrid performance, we can examine the smoothing factor of the relaxation error operator $S_j$, which offers a sharp prediction of actual multigrid performance. 
We define the LFA smoothing factor for   $S_j$ as
\begin{equation}
	\mu_{\rm loc}(S_j) := \max_{\boldsymbol{\theta} \in T^{\rm{H}}}\{\rho(\widetilde{S}_j(\boldsymbol{\theta}))\},
\end{equation}
where {the matrix} $\widetilde{S}_j(\boldsymbol{\theta})$ is the symbol of $S_j$ and $\rho(\widetilde{S}_j(\boldsymbol{\theta}))$ standards for its spectral radius.  Since $\mu_{\rm loc}(S_j)$ is a function of $\omega\in\IR$, we can minimize $\mu_{\rm loc}(S_j)$ to obtain fast convergence. We define
\begin{equation}
	\mu_{\rm opt} :=\min_{\omega}\mu_{\rm loc}(S_j),\qquad {\omega_{\rm opt}:=\argmin_{\omega} \mu_{\rm loc}(S_j)}.
\end{equation}
Let $\eta_j=\lambda_jh^2$. Following \cite{CH2021addVanka}, it can be shown that the  stencil of the element-wise patch $M_{e,j}$ is  
	\begin{equation}\label{eq:general-stencil-element}
		M_{e,j} = \frac{h^2}{4} 
		\begin{bmatrix}
			c & 2b  &c\\
			2b & 4a  &2b\\
			c & 2b  &c
		\end{bmatrix},
	\end{equation} 
	where  $h$ is the spatial mesh step size and 
	\begin{eqnarray*}
		a=\frac{1}{4}\left(\frac{1}{2+\eta_j}+\frac{2}{4+\eta_j}+\frac{1}{6+\eta_j}\right),\,\,
		b=\frac{1}{4}\left(\frac{1}{2+\eta_j}-\frac{1}{6+\eta_j}\right),\,\,
		c=\frac{1}{4}\left(\frac{1}{2+\eta_j}-\frac{2}{4+\eta_j}+\frac{1}{6+\eta_j}\right).
\end{eqnarray*}	
	 With the  stencil \eqref{eq:general-stencil-element}, we can explicitly form the global sparse smoother matrix rather than solving each subproblems in the usual Vanka setting. Furthermore, the symbols of $L_j$ and $M_{e,j}$ are 
	\begin{equation*}
		\widetilde{L}_j =\frac{1}{h^2}(4+\lambda_j h^2-2\cos \theta_1 -2\cos \theta_2),\quad  \widetilde{M}_{e,j} =h^2(a+b\cos \theta_1 +b\cos \theta_2 +c\cos\theta_1\cos \theta_2).
	\end{equation*}
	Thus, we get the following complicated symbol expression
	\begin{equation}\label{eq:smoothing-form}
		\widetilde{S}_j=1-\omega \widetilde{M}_{e,j} \widetilde{L}_j = 1-\omega (a+b\cos \theta_1 +b\cos \theta_2 +c\cos\theta_1\cos \theta_2)(4+\lambda_j h^2-2\cos \theta_1 -2\cos \theta_2).
	\end{equation}
With complex numbers $a,b,c$ and $\lambda_j$,  it is very difficult to pursue the analytically optimal smoothing factor for $S_j$ by {minimizing $\mu_{\rm loc}(S_j)$} over $\omega\in\IR$. Hence, we give an upper bound on the smoothing factor, which is close to the optimal smoothing factor that we numerically obtained from LFA. 
	
	When $\lambda_j =0$, we denote $M_{e,j}$ by  $M_0$. Here, we assume the shifts satisfy $\lambda_j=O(1/h)$, then $\lambda_j h^2=O(h)$ is very small as $h$ is refined. It follows that  $\widetilde{M}_{e,j} \widetilde{L}_j \approx \widetilde{M}_0 \widetilde{A}$, which is insensitive to the complex-shift $\lambda_j$. It is natural to have the following smoothing factor estimate result.

	\begin{theorem} \label{thm:general-upper-bound}
		Assume $\lambda_j=O(1/h)$ and define the smoothing operator $S_j = I-\omega M_{e,j} L_j$. Then,
		\begin{equation}\label{eq:upper-bound-smoothing}
			\mu_{\rm loc}(S_j) \leq \max_{\boldsymbol{\theta} \in T^{\rm{H}}} | 1- \omega  \widetilde{M}_0 \widetilde{A} |+\max_{\boldsymbol{\theta} \in T^{\rm{H}}} |\omega \lambda_j \widetilde{M}_{e,j} | {=:\phi_0(\omega)+\phi_j(\omega)}.
		\end{equation}
	\end{theorem}
	\begin{proof}
	With	$S_j=  I-\omega M_{e,j} A- \omega \lambda_j M_{e,j}$, the conclusion obviously follows from the inequality
		\begin{equation*}
			| \widetilde{S}_j | = |1- \omega \widetilde{M}_e \widetilde{A}  - \omega \lambda_j   \widetilde{M}_{e,j} | \leq | 1- \omega   \widetilde{M}_0 \widetilde{A}|+|\omega \lambda_j  \widetilde{M}_{e,j} |. \qedhere
		\end{equation*}
	\end{proof}	
	In \cite{CH2021addVanka}, it was shown that {$\min_{\omega} \phi_0(\omega)=\frac{7}{25}\approx 0.280$} with $\omega=\frac{24}{25}$. {The second term $\phi_j(\omega)$ in (\ref{eq:upper-bound-smoothing}) can be ignored since $\phi_j(\omega)=O(\lambda_j h^2)=O(h) \ll 0.280$}. In other words,  the optimal smoothing factor for $\mu_{\rm loc}(S_j)$ in Theorem \ref{thm:general-upper-bound}  is about $0.280$. From our LFA numerical tests, it indeed shows that $\omega=\frac{24}{25}$  gives the ({approximately}) optimal smoothing factor. 
	{We highlight that the above estimate may become less useful without the key assumption $\lambda_j=O(1/h)$, which is beyond our scope.}

{As is well-known that LFA smoothing factor often offers a sharp prediction of two-grid convergence factor.   Let  $\nu$ be the number of smoothing steps in multigrid.  We  numerically optimize the LFA two-grid convergence factor \cite{wienands2004practical,trottenberg2000multigrid} with  $\nu=1$}, and then use this optimal parameter  to test LFA two-grid convergence factor as a function of  $\nu$, shown in Table \ref{tab:mu-rho-results-256}. For comparison, we also include  damped Jacobi smoother.  From Table \ref{tab:mu-rho-results-256}, we see that $\rho(1)=\mu_{\rm opt}$ and the optimal parameter is the same as we predicted $\omega=\frac{24}{25}$  from standard Laplacian without shifts, and the Vanka smoother significantly outperforms the Jacobi smoother, which are also confirmed by numerical examples.
	
	\begin{table}[H] 
		\caption{LFA predicted two-grid convergence factor, $\rho(\nu)$,  and smoothing factor with  optimal parameter $\omega$  that minimizes $\rho(1)$ obtained from LFA for Example 1 (Heat equation) with  $\lambda_1$, the first eigenvalue of $B$. $h=\tau=\frac{1}{256}$.}
		\centering
		\begin{tabular}{|l||c|c|cccc|}
			\hline
			&$\omega_{\rm opt}$  & $\mu_{\rm opt}$  &$\nu=1$  & $\nu=2$   &$\nu=3$   & $\nu=4$  \\ \hline
			$\rho_{\mathrm{Jacobi}}$     &0.80        &0.600      &0.600       &0.360      & 0.216   & 0.137   \\ \hline  
			$\rho_{\mathrm{Vanka}}$   &0.96        &0.280      &0.280     &0.116      &0.082   &0.064   \\ \hline  
			
		\end{tabular}\label{tab:mu-rho-results-256}
	\end{table}
	
	\begin{remark}
		The optimal parameter $\omega_{\rm opt}$ and smoothing factor $\mu_{\rm opt}$ given in Table \ref{tab:mu-rho-results-256}  {work}   uniformly   for all eigenvalues of $B$ and leads to same convergence factors. We omit the duplicated results here. 
	\end{remark}

\section{Numerical examples}
In this section, we present several numerical examples (on a unit square domain) to illustrate the effectiveness of our method. All simulations are  implemented with MATLAB on {a Dell Precision 5820 Workstation with Intel(R) Core(TM) i9-10900X CPU@3.70GHz and 64GB RAM},
where the serial CPU times (in seconds) are estimated by the timing functions \texttt{tic/toc}.
In our multigrid solver, we use coarse operator from re-discretization with $2h$, full weighting restriction and linear interpolation operators, W cycle with 1-pre and no post smoothing iteration, coarsest mesh step size $h_0=1/8$, and stopping tolerance $10^{-8}$ based on reduction in relative residual norms. 
We compare the damped \textit{Jacobi} smoother with $\omega=4/5$
and our \textit{Vanka} smoother with $\omega=24/25$. Since our used time schemes are unconditionally stable, we will simple choose the time step size $\tau=h$. 
\subsection{Example 1 (Heat equation).}
In our first example, we test a PinT direct solver \cite{liu2021well} for solving 2D heat equation, where the time discretization matrix $B$ (with a time step size $\tau=1/n$) is given by
\begin{equation}\label{B} \small
	B=\frac{1}{\tau}\begin{bmatrix}
		0 &\frac{1}{2} & & & \\
		-\frac{1}{2} &0 &\frac{1}{2}  & & \\
		&\ddots &\ddots &\ddots &\\
		&  &-\frac{1}{2}  &0   &\frac{1}{2}\\
		&  &   &-1   &1\\
	\end{bmatrix}\in\IR^{n\times n}. 
\end{equation}
It was shown in \cite{liu2021well} that  $|\lambda_j|<(n+\sqrt{n}/\sqrt{2})$ for all $j$, which implies $\lambda_j=O(1/h)$ since $h=\tau=1/n$.
Figure \ref{ex1_fig1} shows our proposed Vanka smoother is about 3 times faster than the Jacobi smoother, where the observed uniform convergence rates match well with the LFA prediction in Table \ref{tab:mu-rho-results-256}.
 \begin{figure}[htp!]
	\centering
	\includegraphics[width=0.49\textwidth,frame]{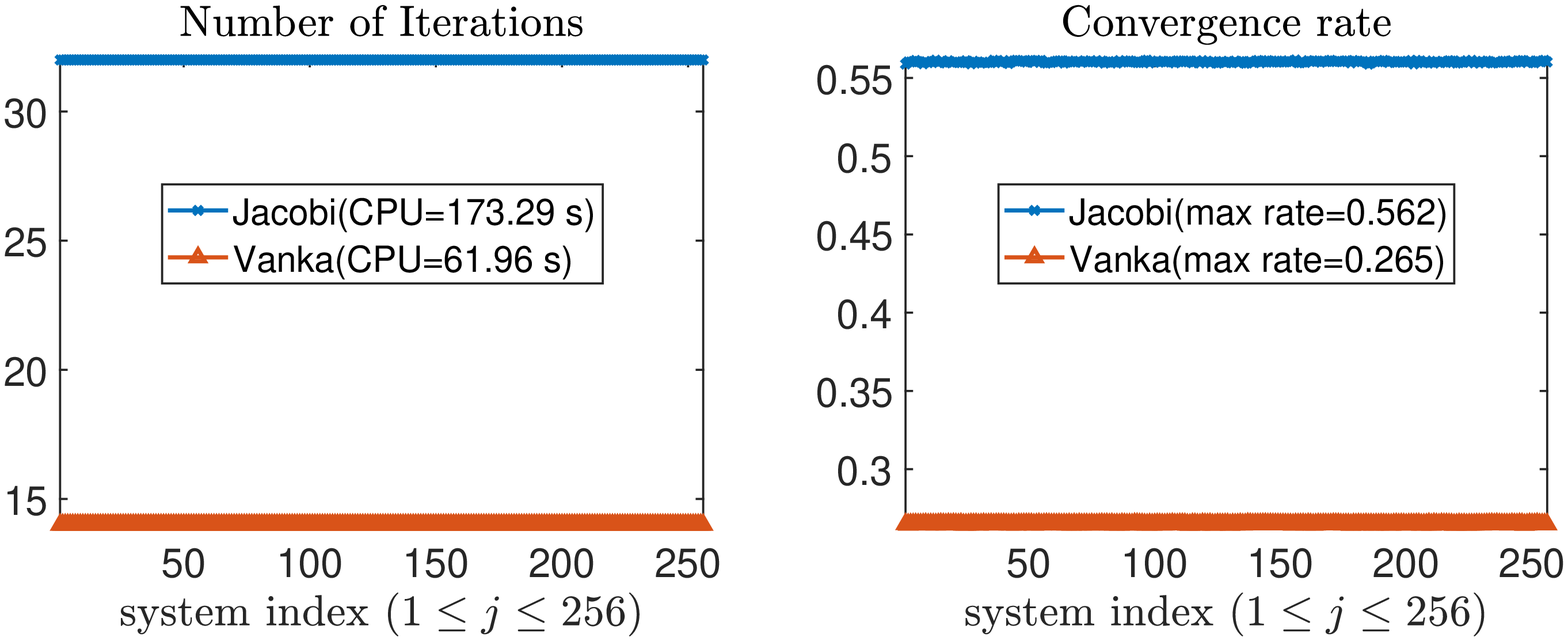}  
	\includegraphics[width=0.49\textwidth,frame]{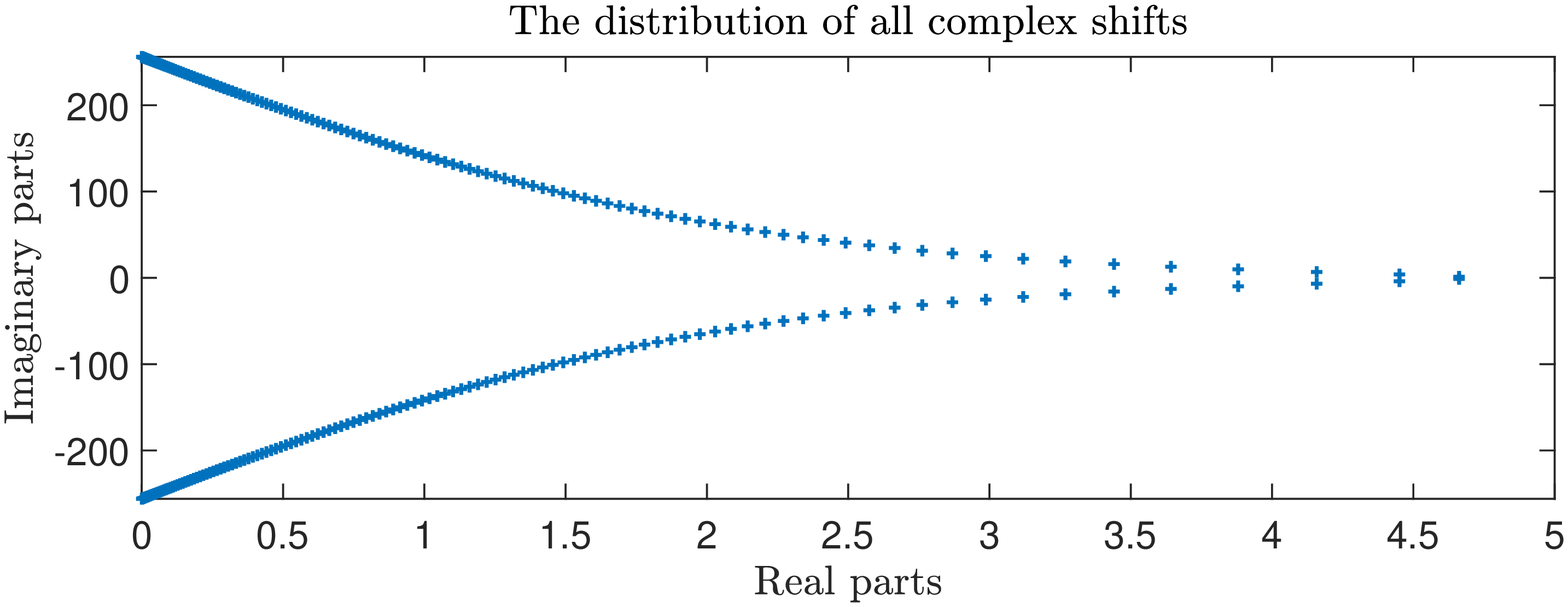}  
	\caption{Ex 1. Left plot: comparison of W-cycle iteration numbers and convergence rates with respect to different shifts ($h=1/256,\tau=1/256$); Right plot: distribution of complex shifts (or eigenvalues of  $B$).} \label{ex1_fig1}
\end{figure} 
  
\subsection{Example 2 (Backward heat conduction problem).}
In our second example, we test a PinT quasi-boundary value method  \cite{liu2021fast} for backward heat conduction problem, where the time discretization matrix $B$ (with a time step size $\tau=1/n$) reads
\begin{equation}\label{B2} \small
	B=\frac{1}{\tau}\bmt
	1& 0 &\cdots &0 &1/\beta\\ 
	-1 &  1  & \cdots &0 &0\\
	\vdots&\ddots &\ddots &\ddots &\vdots\\
	0& 0 &-1& 1& 0\\
	0&0 & 0 & -1 &1
	\emt\in\IR^{(n+1)\times (n+1)}, 
\end{equation}
where  $\beta=\delta>0$ is a small regularization parameter determined by the noise level $\delta=0.01$. Figure \ref{ex2_fig1} shows the same convergence rates, although the eigenvalues of $B$ leads to very different shifts.
 \begin{figure}[htp!]
	\centering
	\includegraphics[width=0.49\textwidth,frame]{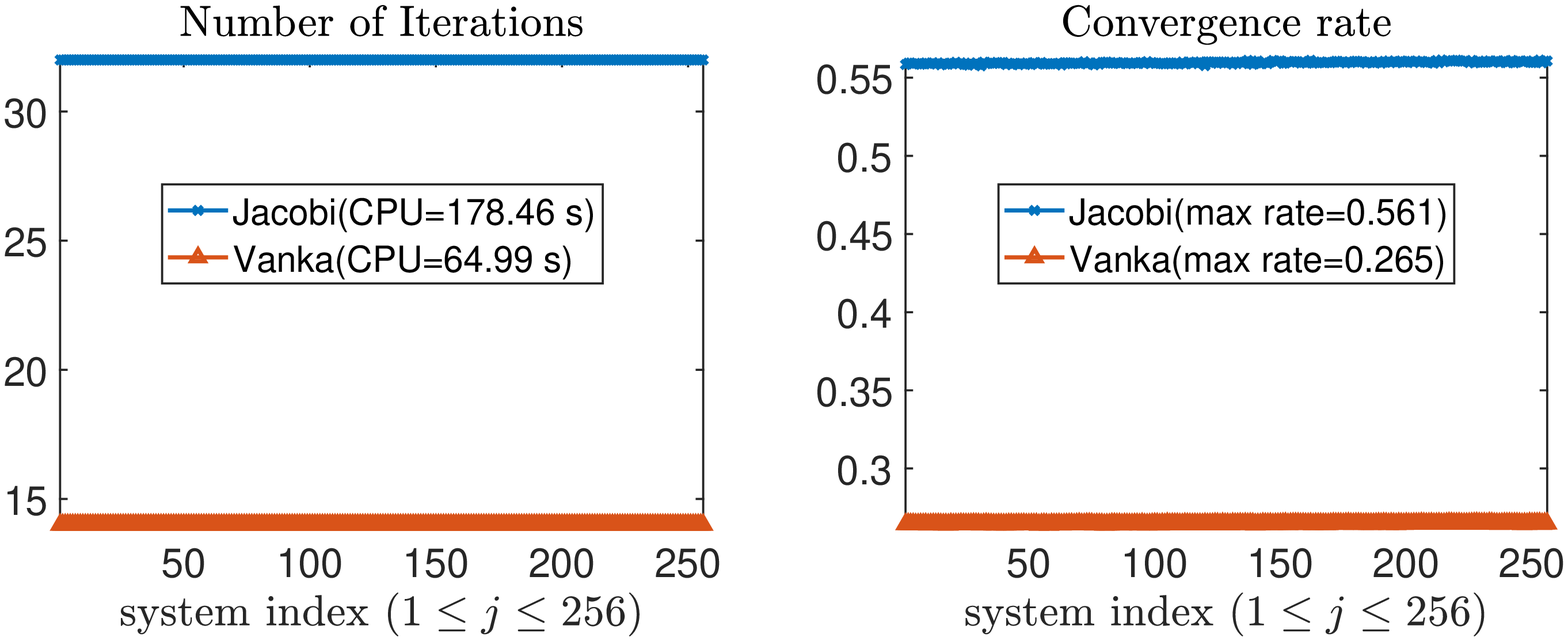}  
	\includegraphics[width=0.49\textwidth,frame]{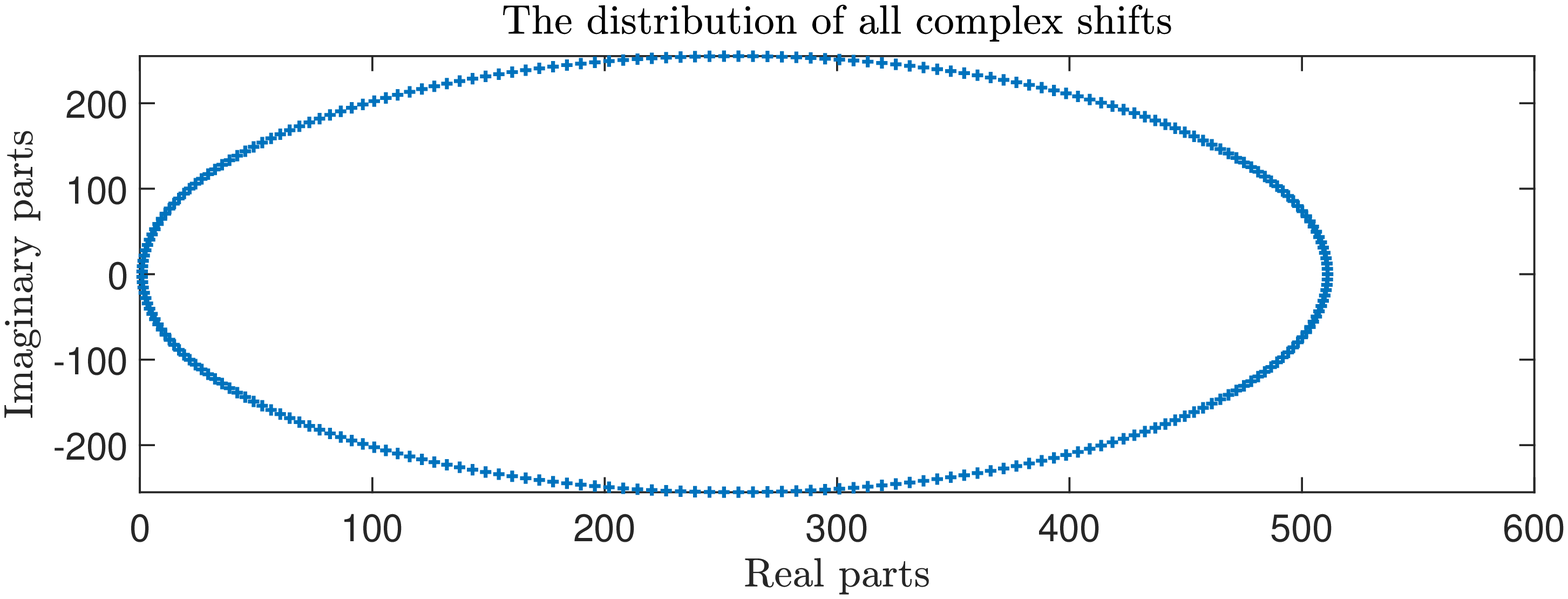}  
	\caption{Ex 2. Left plot: comparison of W-cycle iteration numbers and convergence rates with respect to different shifts ($h=1/256,\tau=1/256$); Right plot: distribution of complex shifts (or eigenvalues of  $B$).} \label{ex2_fig1}
\end{figure} 
\subsection{Example 3 (Isotropic Helmholtz equation).}
In our third example, we test a sequence of 2D isotropic Helmholtz equation
with the shifts given by $\lambda_j=-j^2(1-0.5\Ii)$ with wavenumbers $j=1,2,\cdots,128$ satisfying $jh\le 1/2$ to avoid the pollution effect \cite{Hocking2021}. Here $\Ii=\sqrt{-1}$  denotes the imaginary unit.
Let `Jacobi-C' denotes the damped Jacobi smoother with optimal complex relaxation parameter as proposed in \cite{Hocking2021}.
Figure \ref{ex3_fig1} shows our proposed Vanka smoother is about 3 times faster than both Jacobi and Jacobi-C smoother, where the Jacobi-C smoother outperforms the Jacobi smoother for large wavenumbers.
The gradual deterioration of convergence rates for larger shifts is expected due to their negative real parts.
The advantage of Jacobi-C smoother indicates our Vanka smoother may be further improved with carefully chosen complex relaxation parameter, which will be left as our future work.
 \begin{figure}[htp!]
	\centering
	\includegraphics[width=0.49\textwidth,frame]{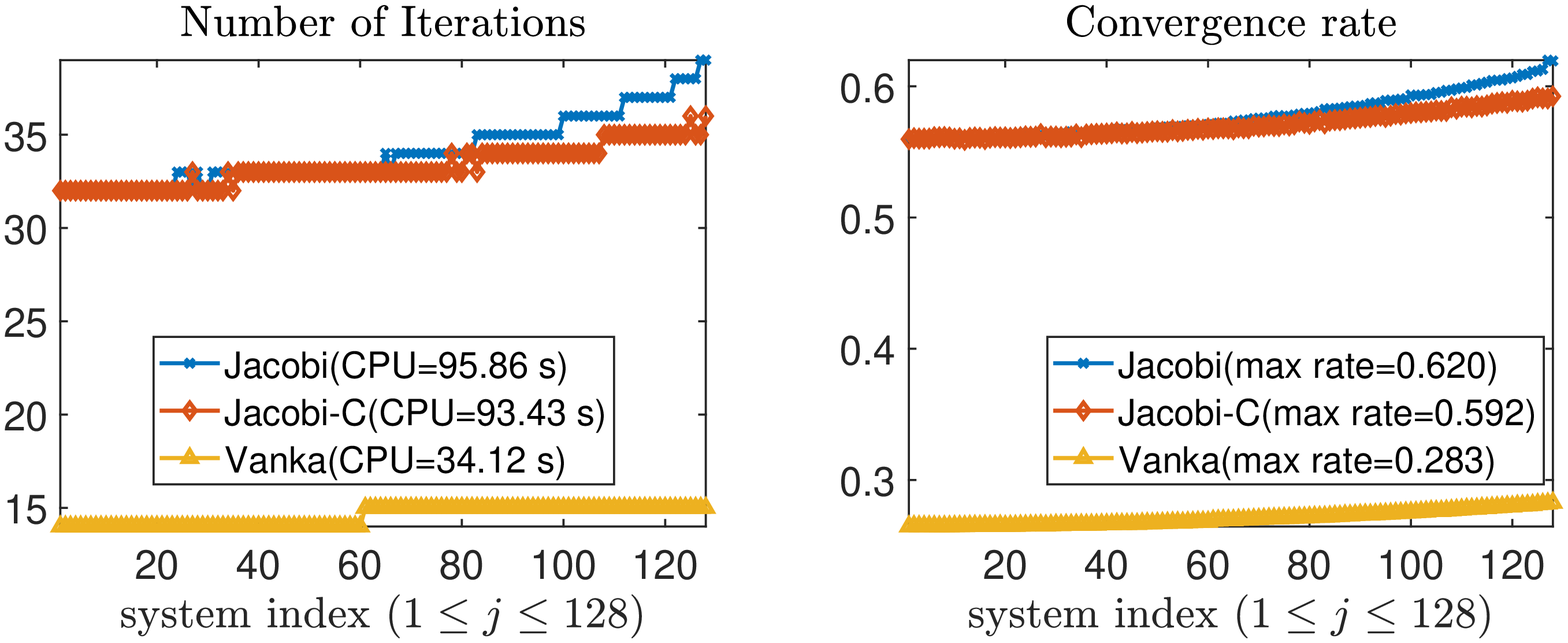}  
	\includegraphics[width=0.49\textwidth,frame]{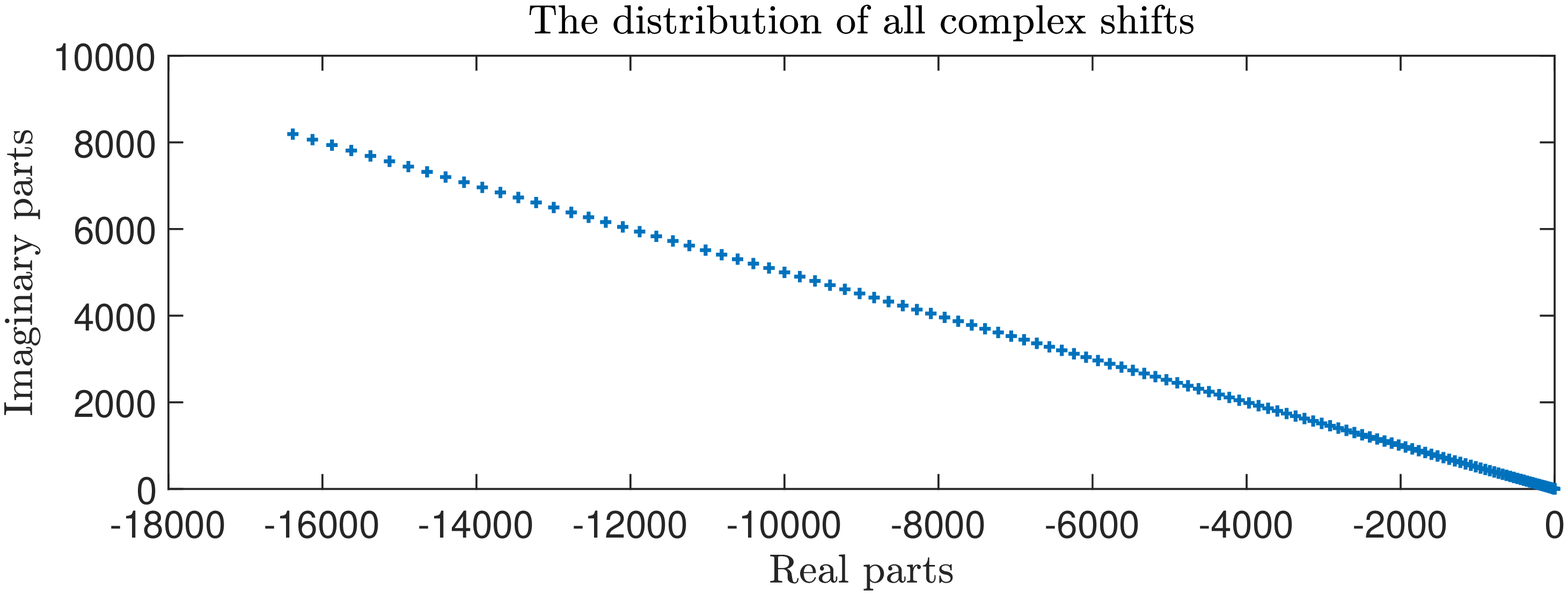}  
	\caption{Ex 3. Left plot: comparison of W-cycle iteration numbers and convergence rates with respect to different shifts ($h=1/128$); Right plot: distribution of complex shifts.} \label{ex3_fig1}
\end{figure}
\section{Conclusions}
 In this paper we have proposed and analyzed an additive element-wise Vanka smoother for complex-shifted Laplacian systems that stem from a class of diagonalization-based PinT algorithms.
 Numerical results with various applications confirmed our theoretical outcomes based on LFA techniques.
 It is also possible to generalize our method to similar complex-shifted linear systems  arising in the Laplace transform-based parallelizable contour integral method (see e.g., \cite{sheen2003parallel}). 
{\small
\bibliographystyle{elsarticle-num}
\bibliography{LFA,ISPpint,DirectPinT,inversePDE,waveControl,waveControl2}
}

\end{document}